\documentclass[a4paper,10pt]{amsart}
 \usepackage{amsfonts,mathrsfs}
 \usepackage{amsthm}
 \usepackage{amssymb}
 \usepackage{enumerate}
 \usepackage{enumitem}
 \usepackage{tikz-cd}
 \usepackage[hidelinks]{hyperref}
 \usepackage{cleveref}
 \usepackage{ulem}
\theoremstyle{definition}
\newtheorem{Def}{Definition}[section]
\newtheorem{ex}[Def]{Example}
\newtheorem{rem}[Def]{Remark}

\theoremstyle{plain}
\newtheorem{prop}[Def]{Proposition}
\newtheorem{thm}[Def]{Theorem}
\newtheorem*{thm*}{Theorem}
\newtheorem{lem}[Def]{Lemma}

\newtheorem{cor}[Def]{Corollary}
\newtheorem*{cor*}{Corollary}

\newtheorem*{con*}{Conjecture}
\newtheorem*{qu}{Question}
\newtheorem*{verm*}{Vermutung}

\usepackage{mathrsfs}

\newcommand{\Pic}{\operatorname{Pic}}

\newcommand{\Sym}{\operatorname{Sym}}

\newcommand{\Proj}{\operatorname{Proj}}

\newcommand{\codim}{\operatorname{codim}}

\newcommand{\lk}{\operatorname{lk}} 

\newcommand{\cE}{{\mathcal E}}

\newcommand{\cO}{{\mathcal O}}

\newcommand{\cV}{{\mathcal V}}

\newcommand{\D}{{\mathbb D}}
\newcommand{\B}{{\mathbb B}}
\newcommand{\G}{{\mathbb G}}
\newcommand{\C}{{\mathbb C}}
\newcommand{\R}{{\mathbb R}}
\newcommand{\pp}{\mathbb{P}}

\newcommand{\Z}{{\mathbb Z}}

\title[Real-fibered morphisms and del Pezzo surfaces]{Real-fibered morphisms of del Pezzo surfaces and conic bundles}
\author{Mario Kummer}
\address{Technische Universit\"at Dresden, Germany} 
\email{mario.kummer@tu-dresden.de}

\author{C\'edric Le Texier}
\address{Universitetet i Oslo, Norway}
\email{cedricle@math.uio.no}

\author{Matilde Manzaroli}
\address{Universitetet i Oslo, Norway}
\email{manzarom@math.uio.no}

\thanks{The research of the last two authors is funded by the Trond Mohn Stiftelse (TMS) project “Algebraic and topological cycles in complex and tropical geometry”}

\DeclareMathOperator{\Bl}{Bl}

\newcommand{\comment}[1]{}

\begin{document}
 

\begin{abstract}
It goes back to Ahlfors that a real algebraic curve admits a real-fibered morphism to the projective line if and only if the real part of the curve disconnects its complex part. Inspired by this result, we are interested in characterising real algebraic varieties of dimension $n$ admitting real-fibered morphisms to the $n$-dimensional projective space. We present a criterion to classify real-fibered morphisms that arise as finite surjective linear projections from an embedded variety which relies on topological linking numbers. We address special attention to real algebraic surfaces. We classify all real-fibered morphisms from real del Pezzo surfaces to the projective plane and determine which such morphisms arise as the composition of a projective embedding with a linear projection. Furthermore, we give some insights in the case of real conic bundles. 

\end{abstract}

 \maketitle
 \tableofcontents

\section{Introduction}
This work concerns the study of \textit{real-fibered} morphisms from real algebraic varieties to projective spaces of same dimension. Apart from the case of real algebraic curves, the topology of the real part of higher dimensional real algebraic varieties admitting real-fibered morphisms is bound to be a disjoint union of spheres and real projective spaces. We mainly focus on real del Pezzo surfaces and real conic bundles whose real classification is well known. The study of real algebraic varieties dates back to the 19th century. One of the first significant results was the classification of real cubic surfaces presented in \cite{Schl63}. Then, in \cite{Zeut74}, the study of real plane algebraic curves of degree $4$ and their bitangents was carried out (which is equivalent to the study of real del Pezzo surfaces of degree $2$). The first systematic study of real algebraic varieties was pursued by Harnack, Klein, Hilbert and Comessatti \cite{Harn76,Hilb02,Come13,Come14,Klei73}. In particular, Comessatti classified real rational algebraic $\mathbb{R}$-minimal surfaces. Moreover, since we can obtain any real rational surface as a sequence of real blow-ups of a real rational $ \mathbb{R} $-minimal surface, Comesatti's approach in \cite{Come14} leads to a complete classification of real del Pezzo surfaces.\\
\par
Let $X$ be a non-singular algebraic variety of dimension $n$ (by variety we will always mean an integral and separated scheme of finite type over $\R$). 
In this article, we suppose that all varieties have non-empty real part unless otherwise stated.
\begin{Def}
\label{real-fibered}
Let $X$ and $Y$ be non-singular algebraic varieties of dimension $n$. We say that a real morphism $f : X \rightarrow Y$ is \emph{real-fibered}, if {$X(\R)$ is non-empty} and $f^{-1}(Y(\mathbb{R}))=X(\mathbb{R})$.
\end{Def} 
As already mentioned, we are particularly interested in real-fibered morphisms $X\to\pp^n$ where $\pp^n$ is the scheme $\pp^n_\R=\textrm{Proj}(\R[x_0,\ldots,x_n])$. 
According to a result by Ahlfors \cite[\S4.2]{Ahlf50}, it is known which projective irreducible smooth curves $C$ admit a real-fibered morphism $C\to\pp^1$. This is the case if and only if $C$ is of \emph{Type I} or \emph{separating} in the sense that its real points $C(\mathbb{R})$ disconnect its set of complex points $C(\mathbb{C})$. Note that $C$ is separating whenever $C$ is an \emph{$M$-curve}, i.e. the number $r$ of connected components of $C(\R)$ equals $g+1$ where $g$ is the genus of $C$. On the other hand, if $C$ is separating then $r$ has the same parity as $g+1$. Separating curves and their real-fibered morphisms to $\pp^1$ have been studied by several authors, see for example \cite{gabard, cophuis, cop, ore, sep}.
While any separating curve $C$ admits real-fibered morphisms to $\pp^1$ of arbitrary large degree, the situation is much more rigid for varieties of higher dimension. This is mainly due to the fact that for any real-fibered morphism $X\to Y$ of smooth varieties the restriction to the real parts is an unramified covering map \cite[Thm.~2.19]{realfib}. Among others, this implies that for any smooth variety $X$ of dimension $n\geq2$, the topology of $X(\R)$ already determines the degree of any real-fibered morphism $X\to\pp^n$. More precisely, if $X$ is a smooth variety of dimension $n\geq2$ and $f:X\to\pp^n$ a real-fibered morphism, then $X(\R)$ is homeomorphic to a disjoint union of $s$ spheres and $r$ real projective spaces such that $\deg(f)=2s+r$ (or $X(\R)$ is empty).
However, there is no topological characterisation known, similar to Ahlfors' above result on curves of Type I, of those $n$-dimensional varieties that admit a real-fibered morphism $X\to\pp^n$. A possible approach to extend the notion of being Type I to all varieties is presented in \cite{dividingsurfaces}, where Viro introduces the definition of \textit{bound in complexification} for a smooth $n$-dimensional variety $X$: this is the case if $X(\R)$ realises the trivial element in the homology group $H_{n}(X_{\mathbb{C}}; \mathbb{Z}/2\mathbb{Z})$. In this article, we go towards a different direction and we give criteria to characterise smooth $n$-dimensional varieties admitting real-fibered morphisms $X\to\pp^n$. 
Moreover, we completely characterise \emph{del Pezzo surfaces} which admit real-fibered morphisms to $\pp^2$. 

In the following, for a non-singular algebraic variety $X$ of dimension $n$, we write \emph{real Picard group} for the Picard group of $X_{\R}$, respectively Picard group for the Picard group of $X_{\C}$. We say that a morphism of real varieties is a \emph{finite} real-fibered morphism if it is real-fibered and a finite morphism in the sense of \cite[p.~84]{Hart77}. We first present the following characterisation of {finite} real-fibered morphisms $X\to\pp^2$ from a del Pezzo surface $X$.
\begin{thm}
\label{thm : RealFibDelPezzo}
 Let $X$ be a del Pezzo surface such that each connected component of $X(\R)$ is homeomorphic to either the sphere or the real projective plane. There is a finite real-fibered morphism $X\to\pp^2$ if and only if we have one of the following:
 \begin{enumerate}
  \item $X$ has real Picard rank $1$;
  \item $X$ is a conic bundle of real Picard rank $2$;
  \item $X$ is the blow-up of one of the above surfaces at one or two real points.
 \end{enumerate}
\end{thm}

Since the blow-up at a pair of complex conjugate points is always real-fibered, we obtain the following.

\begin{cor}
\label{cor:non-finite}
Let $X$ be a del Pezzo surface with $X(\R)\neq\emptyset$. 
There is a (possibly non-finite) real-fibered morphism $X\to\pp^2$ if and only if each connected component of $X(\R)$ is homeomorphic to either the sphere or the real projective plane.
\end{cor}

A concept closely related to real-fibered morphisms is the notion of hyperbolic varieties.

\begin{Def}
\label{hyperbolic}
Let $X\subset\pp^N$ be an embedded variety and $E\subset\pp^N$ a linear subspace  of dimension $d=\codim(X,\pp^N)-1$ with $E\cap X=\emptyset$. Then $X$ is \emph{hyperbolic} with respect to $E$ if for all linear subspaces $E'\supset E$ of dimension $d+1$ we have that $E'\cap X$ consists only of real points. 
\end{Def}

Note that $X$ is hyperbolic with respect to $E$ if and only if the linear projection $\pi_E: X\to\pp^{\dim X}$ from center $E$ is real-fibered.
Hyperbolic embeddings of curves were studied for instance in \cite{ore, sep}. For example it was shown in \cite{sep} that any embedding of a separating curve via a complete linear system of large enough degree is hyperbolic. Hyperbolic curves also played an important role in the recent classification of maximally writhed real algebraic links \cite{maxwrith}. For higher dimensional varieties we give the following characterisation of hyperbolic varieties that allows us to reduce the problem to smaller dimensions. From now on, we will write $X(\R ) \simeq s S^k \sqcup r \R \pp^k$ in order to express that the real part is homeomorphic to the disjoint union of $s$ $k$-spheres and $r$ real projective spaces of dimension $k$.

\begin{thm}
\label{thm: alt_linking1}
Let $X \subset \mathbb{P}^n$ be a smooth variety of dimension $k\geq2$. Let $H\subset\pp^n$ be a hyperplane such that $C=X \cap H$ is a smooth $(k-1)$-variety. Assume that each connected component of $X(\mathbb{R})$ contains exactly one connected component of $C(\mathbb{R})$. Moreover, let $E\subset H$ be a linear space of dimension $n-k-1$ with $X\cap E=\emptyset$. Then the following are equivalent:
\begin{enumerate}
 \item $X$ is hyperbolic with respect to $E$.
 \item $X$ satisfies $X(\R ) \simeq s S^k \sqcup r \R \pp^k$ such that $\deg(X)=2s+r$.
	The class of each connected component that is homeomorphic to a real projective space is nontrivial in $H_{k}(\mathbb{P}^{n}(\mathbb{R}); \Z_2)$ and $C\subset H=\pp^{n-1}$ is hyperbolic with respect to $E$.
\end{enumerate}
\end{thm}

This allows us to characterise del Pezzo surfaces which can be embedded in some projective space as a hyperbolic variety.

\begin{thm}
\label{thm : HypDelPezzo}
 Let $X$ be a del Pezzo surface such that each connected component of $X(\R)$ is homeomorphic to either the sphere or the real projective plane. There is an embedding $X\hookrightarrow\pp^n$ such that the image is a hyperbolic variety if and only if we have one of the following:
 \begin{enumerate}
  \item $X$ has real Picard rank $1$;
  \item $X$ is a conic bundle of real Picard rank $2$;
  \item $X$ is the blow-up of one of the above surfaces at one real point.
 \end{enumerate}
\end{thm}

Furthermore, we characterise these embeddings.

\begin{thm}\label{thm:main}
 Let $X\subset\pp^n$ be a smooth nondegenerate real del Pezzo surface embedded via a complete linear system. There exists a linear subspace $E\subset\pp^n$ of codimension $3$ such that $X$ is hyperbolic with respect to $E$ if and only if:
 \begin{enumerate}
  \item $X(\R ) \simeq s S^2 \sqcup r \R \pp^2$;
  \item $\deg(X)=2s+r$;  and
  \item the genus of a hyperplane section on $X$ equals $s+r-1$.
 \end{enumerate}
 In this case we further have $r\in\{0,1\}$ and $n=s+2$.
\end{thm}

Part $(2)$ of the Theorems \ref{thm : RealFibDelPezzo} and \ref{thm : HypDelPezzo} motivates the question of which real conic bundles $X$ (over $\pp^1$) with real Picard rank 2 admit a real-fibered morphism to $\pp^2$, respectively which ones admit a hyperbolic embedding.
In order to treat this question, we will consider those surfaces $X$ as the zero set of a section of $\mathcal{O}_{\pp (\mathcal{E})} (2)$, for $\pp (\mathcal{E} )$ a projective plane bundle over $\pp^1$.
This will allow us to construct hyperbolic conic bundles with an arbitrary large number of components homeomorphic to a sphere.
Finally, for all pairs $(s,r)$ of nonnegative integers we decide whether there exist a smooth hyperbolic surface with real part homeomorphic to $s S^2 \sqcup r \R \pp^2$, except for the case $s=2 , r \geq 2$. We provide a construction when $s\geq 3 , r\geq 0$ using the tautological embedding of the projective bundle $\pp (\mathcal{E})$. 

The paper will be organised as follows.
We start by recalling several facts and notations about real del Pezzo surfaces 
in \Cref{sec:pre}.
We then give in \Cref{sec:nec} a classification of ample divisors $D$ on real del Pezzo surfaces $X$ satisfying some necessary conditions for the associated morphism $f : X \rightarrow \pp^2$ to be finite real-fibered. In \Cref{sec:linking}, we give a criterion for a real variety $X\subset \pp^n$ of dimension $k$ to be hyperbolic with respect to a given linear subspace $E$ of dimension $n-k-1$, in terms of linking numbers (generalising the criterion for real curves given in \cite{sep}). 
This will allow us to prove Theorem \ref{thm: alt_linking1}. In the subsequent Section \ref{sec:hyppezzo} we apply this to prove the Theorems \ref{thm : RealFibDelPezzo}, \ref{thm : HypDelPezzo} and \ref{thm:main} as well as Corollary \ref{cor:non-finite}. For some examples of real del Pezzo surfaces $X$ we construct in \Cref{explicit} an explicit linear subspace $E \subset \pp^n$ of codimension 3 such that $X$ is hyperbolic with respect to $E$. 
Finally, in \Cref{conic_bundles}, we treat the case of hyperbolic minimal conic bundles and the question of which topological types are realisable as real part of a hyperbolic surface.  

\section{Preliminaries and notation}
\label{sec:pre}
Let $X$ be a real smooth irreducible surface. By a surface we always mean a projective and irreducible variety of dimension $2$. We will denote by $-K_{X}$, or simply $-K$ if there is no ambiguity, the anti-canonical class of $X$. 
\begin{Def}
\label{DP_r_minimal}
If $-K$ is ample, then $X$ is a del Pezzo surface of degree $K^{2}$. 
\end{Def}
Let $X$ be a del Pezzo surface. From the complex view point $X_{\mathbb{C}}$ is either $\pp^{2}_{\mathbb{C}}$ blown up in $r$ points in general position, where $r\leq 8$, or $\pp^{1}_{\mathbb{C}} \times \pp^{1}_{\mathbb{C}}$. In the former case one has $K^{2}=9-r$ and $K^{2}=8$ in the latter. In particular $K^{2}\leq 9$. 

\begin{Def}
Let $X$ be a smooth irreducible surface. 
\begin{enumerate}
 \item If every (real) birational morphism from $X$ into a smooth surface is an isomorphism, then we say that $X$ is \emph{minimal (over $\R$)}.
 \item If $X$ is a conic bundle of real Picard rank two, we say that $X$ is a \emph{minimal conic bundle}.
\end{enumerate}
\end{Def}

Every real del Pezzo surface is one of the following or a blow-up of one of the following at a zero dimensional real subvariety: 
\begin{itemize}
\item The projective plane $\pp^2$ which is a del Pezzo surface of degree $9$ whose real part is the real projective plane $\R\pp^2$.
\item The quadric hypersurface $Q^{n,4-n}=\cV(\sum\limits_{i=1}^{n}x_{i}^{2}-\sum\limits_{j=n+1}^{4}x_{i}^{2})$, $n\in\{0,1,2\}$, in $\pp^{3}$ which is a del Pezzo surface of degree $8$. Its real part is empty when $n=0$ and homeomorphic to the sphere $S^2$ resp. the torus $S^1\times S^1$ when $n=1$ or $n=2$ respectively.
\item The direct product $\pp^1\times C$ where $C$ is a smooth rational curve without real points. This is a del Pezzo surface of degree $8$ whose real part is empty.
\item A minimal conic bundle $\D_4$ which is a del Pezzo surface of degree $4$ whose real part is homeomorphic to a disjoint union of $2$ spheres.
\item A minimal conic bundle $\D_2$ which is a del Pezzo surface of degree $2$ whose real part is homeomorphic to a disjoint union of $3$ spheres.
\item A minimal surface $\G_2$ which is a del Pezzo surface of degree $2$ whose real part is homeomorphic to a disjoint union of $4$ spheres.
\item A minimal surface $\B_1$ which is a del Pezzo surface of degree $1$ whose real part is homeomorphic to a disjoint union of $4$ spheres and a real projective plane.
\end{itemize}

Let $X$ be a real surface such that all connected components of $X(\R)$ are homeomorphic to each other. We indicate by $X(a,2b)$ the real surface obtained by blowing up the real surface $X$ in $a$ real points and $b$ pairs of complex conjugated points. If $a=1$, the topology of $X(a,2b)(\R)$ does not depend on the real point at which we blow up. If $a=2$, we denote by $X(2,2b)^2_0$ resp. $X(2,2b)^1_1$ the surfaces obtained by blowing up two real points from the same or two points from different connected components of $X(\R)$ respectively. The case $a\geq3$ will not occur in this work.

\section{Necessary conditions for surfaces}\label{sec:nec}
We first derive some necessary conditions on the ample divisor classes on a surface that arise from a finite real-fibered morphism to $\pp^2$ (\Cref{top_real_morphism}, \Cref{thm:necsurf}). Then, we specify such conditions to the case of del Pezzo surfaces (\Cref{thm:nec}). The case of conic bundles is treated in \Cref{conic_bundles}.
\begin{lem}
\label{top_real_morphism}
 Let $X$ be a smooth irreducible surface and $f: X\to\pp^2$ a real-fibered morphism. Further let $X(\R)\neq\emptyset$. Then $f$ is generically finite to one. Let us denote its degree by $d$. 
Then $X(\R ) \simeq s S^2 \sqcup r \R \pp^2$ such that $d=r+2s$. 
The preimage of a general real line in $\pp^2$ is a smooth irreducible separating curve $C$ such that $C(\R)$ has $s+r$ connected components.
\end{lem}

\begin{proof}
 Since $f$ is real-fibered, the fiber $f^{-1}(f(p))$ for any $p\in X(\R)$ must be finite as any nonfinite variety has nonreal points. Since the dimension of a fiber is upper-semicontinuous \cite[Thm.~11.12]{Ha95}, this shows that $f$ is generically finite to one. Let $U\subset\pp^2$ be the open subset of all $p\in\pp^2$ such that $f^{-1}(p)$ is finite and let $X'=f^{-1}(U)$. The restriction $f':X'\to U$ of $f$ to $X'$ is \emph{quasi-finite} in the sense that each fiber is finite. Moreover, it is proper because $f$ is proper as a morphism between projective varieties. Thus $f'$ is finite by \cite[Thm.~8.11.1]{EGAIV3}. Furthermore, we have $X(\R)\subset X'$ and $\pp^2(\R)\subset U$.
 Thus by \cite[Thm.~2.19]{realfib} the map $X(\R)\to\pp^2(\R)$ obtained by restricting $f$ to the real part of $X(\R)$ is an unramified covering map. This implies that the real part $X(\R)$ is homeomorphic to the disjoint union of $s$ spheres and $r$ projective planes such that $d=r+2s$, see also \cite[Cor.~2.20]{realfib}. By Bertini's lemma \cite[Thm.~6.10]{bertini}, the preimage of a general real line in $\pp^2$ is a smooth irreducible curve $C$ with $C(\R)$ having $r+s$ connected components, one for each connected component of the covering space $X(\R)$. Furthermore, the restriction of $f$ to $C$ is again real-fibered. Thus $C$ is separating, see e.g. \cite[Thm.~2.8]{realfib}.
\end{proof}

\begin{ex}
 By the Noether--Lefschetz Theorem we can approximate the polynomial $x_1^4+x_2^4+x_3^4-x_0^4$ arbitrarily close by a homogeneous polynomial $p\in\R[x_0,x_1,x_2,x_3]$ of degree $4$ such that the zero set $X=\cV(p)\subset\pp^3$ is a smooth surface whose Picard group is generated by the class of a hyperplane section. In particular, there is no morphism $f:X\to\pp^2$ of degree $2$. Further we can choose $p$ in such a way that $X(\R)$ is homeomorphic to a sphere as this is the case for the real zero set of $x_1^4+x_2^4+x_3^4-x_0^4$. Thus by \Cref{top_real_morphism} there is no real-fibered morphism $f:X\to\pp^2$. In particular, for an abstract smooth surface $X$ the criterion $X(\R ) \simeq s S^2 \sqcup r \R \pp^2$, $r,s\in\Z_{\geq0}$, is necessary but not sufficient for admitting a real-fibered morphism  $f:X\to\pp^2$.
\end{ex}

\begin{thm}\label{thm:necsurf}
  Let $f: X\to\pp^2$ be a finite real-fibered morphism from a smooth surface $X$ and $D$ the corresponding ample divisor class. Then we have the following:
 \begin{enumerate}
  \item $X(\R ) \simeq s  S^2 \sqcup r  \R \pp^2$;
  \item $D.D=r+2s$;
  \item $r\leq D.K+4$;
  \item $D.K\equiv r \mod 4$;
  \item $D.L>0$ for all effective divisors $L\subset X_\C$.
 \end{enumerate}
\end{thm}

\begin{proof}
 Part $(1)$ and $(2)$ are part of the previous lemma. In order to prove the inequality of $(3)$, we note that by the preceding lemma the preimage of a general line in $\pp^2$ is a smooth separating curve whose real part has $r+s$ connected components. Its genus $g$ is according to the Adjunction Formula \cite[V, Prop.~1.5]{Hart77}: $$g=\frac{1}{2}(D.(D+K))+1=\frac{1}{2}(r+D.K)+s+1.$$Now Harnack's inequality says that $$r+s\leq g+1=\frac{1}{2}(r+D.K)+s+2$$which implies $r\leq D.K+4$. Since the curve is separating, we furthermore have that $g+r+s=\frac{1}{2}(r+D.K)+2s+r+1$ is odd which implies claim $(4)$. Part $(5)$ is clear because $f$ is finite and $D$ therefore ample. This remains true under a base change to $\C$.
\end{proof}

From \Cref{thm:necsurf} and the Kodaira Vanishing Theorem, one obtains the following statement for del Pezzo surfaces. The Kodaira Vanishing Theorem is used to prove the right-hand inequality in $(3)$ of \Cref{thm:nec}.
This allows us to narrow down the ample divisor classes that can possibly arise as the pull-back of a hyperplane section under a finite real-fibered morphism $X\to\pp^2$ to a finite list.

\begin{cor}\label{thm:nec}
 Let $f: X\to\pp^2$ be a finite real-fibered morphism from a del Pezzo surface $X$ and $D$ the corresponding ample divisor class. Then we have the following:
 \begin{enumerate}
  \item $X(\R ) \simeq s S^2 \sqcup r \R \pp^2$;
  \item $D.D=r+2s$;
  \item $r\leq D.K+4\leq r+2s$;
  \item $D.K\equiv r \mod 4$;
  \item $D.L>0$ for all lines $L\subset X_\C$.
 \end{enumerate}
 In particular, there are only finitely many possibilities for such $D$.
\end{cor}

\begin{proof}
In order to prove the missing inequality of $(3)$, we will compute $\ell(D)$ using the Riemann--Roch Theorem for surfaces \cite[V, Thm.~1.6]{Hart77}. For this we note that by \cite[V, Cor.~3.5]{Hart77} the arithmetic genus of $X$ is zero. Furthermore, because $D$ and $-K$ are ample, we have that $D-K$ is ample as well and thus the Kodaira Vanishing Theorem \cite[V, Rem.~7.15]{Hart77} implies that $h^i(D)=0$ for $i>0$. Therefore, we have $$\ell(D)=\frac{1}{2}D.(D-K)+1$$by Riemann--Roch. Since $D$ comes from a morphism to $\pp^2$, we have $\ell(D)\geq3$. Together with $(2)$ this implies $D.K+4\leq r+2s$. Finally, it follows from the Hodge Index Theorem \cite[Rem~1.9.1]{Hart77} and the fact that $-K$ is ample that there can be only finitely many $D$ satisfying $(2)$ and $(3)$.
\end{proof}

For each real del Pezzo surface $X$ whose real part $X(\R)$ consists of connected components that are homeomorphic to spheres and real projective planes only, we determined all divisor classes on $X$ that satisfy all requirements of \Cref{thm:nec} via a brute-force search. The result is listed in \Cref{tb:list}. In this table, we use the following notation regarding generators of the real Picard group of real del Pezzo surface:
\begin{itemize}
\item $\textrm{Pic}(\mathbb{D}_{2})=\langle-K, F \rangle$, where $F$ denotes the class of a fiber (when regarding $\D_2$ as a conic bundle $\D_2\to\pp^1$) and $-K.F=2$;
\item $\textrm{Pic}(\mathbb{D}_{2}(1,0))=\langle-K, \tilde{F}, E \rangle$, where $E$ and $\tilde{F}$ are $(-1)$-curves and any two distinct generators have intersection equal to one;
\item $\textrm{Pic}(\mathbb{G}_{2}(1,0))=\langle-K, E \rangle$, where $E$ is a $(-1)$-curve and $-K.E=1$.
\end{itemize}

In the following example, we carry this computation out for the real del Pezzo surface $\mathbb{D}_2$. The other cases can be treated analogously but we do not write down the explicit calculations here. 

\begin{ex}
 The real part of the del Pezzo surface $\mathbb{D}_2$ of degree two consists of three connected components that are homeomorphic to a sphere \cite[Cor.~4.3]{russo}. Thus we have $s=3$ and $r=0$. The complexification $(\mathbb{D}_2)_\C$ of $\mathbb{D}_2$ is the blow-up of $\mathbb{P}^2$ at seven points. Thus the Picard group of $(\mathbb{D}_2)_\C$ is the free abelian group generated by the pull-back of $\cO_{\pp^2}(1)$ and the classes of the seven exceptional divisors $E_{1},..,E_{7}$. In this basis the complex conjugation on $\Pic(\mathbb{D}_2)_\C$ is given by the matrix $$\begin{pmatrix}
                          4& 3& 1& 1& 1& 1& 1& 1\\-3& -2& -1& -1& -1& -1& -1& -1\\-1& -1& -1& 0& 0& 0& 0& 0\\-1& -1& 0& -1& 0& 0& 0& 0\\-1& -1& 0& 0& -1&      0& 0& 0\\-1& -1& 0& 0& 0& -1& 0& 0\\-1& -1& 0& 0& 0& 0& -1& 0\\-1& -1&      0& 0& 0& 0& 0& -1
                         \end{pmatrix}$$
 where the first coordinate corresponds to the pull-back of $\cO_{\pp^2}(1)$, see \cite[Exp.~2]{russo}. The real Picard group of $\mathbb{D}_2$ consists of those divisor classes of $(\mathbb{D}_2)_\C$ that are fixed under this involution. Thus it is generated by the two divisor classes $F=(1,-1,0,0,0,0,0,0)$ and $K=(-3,1,1,1,1,1,1,1)$ where the latter is the canonical divisor class. We observe that $F.F=0$, $F.K=-2$ and $K.K=2$. Assume that the divisor $D=hF-lK$ for $h,l\in\Z$ satisfies the conditions from \Cref{thm:nec}. Condition $(2)$ says that $2(2h+l)l=6$. Here are the only pairs $(h,l)$ of integers that satisfy this condition: $$(1,-3),(-1,-1),(1,1),(-1,3).$$
 Finally, condition $(5)$ applied to $L=E_6$ for instance rules out the possibility of $l<0$, so we are left with $(h,l)=(1,1)$ and $(h,l)=(-1,3)$. Alternatively, this is also implied by conditions $(3)$ and $(4)$ which show that $h+l\in\{0,2\}$.
\end{ex}

\begin{table}[h!]
\begin{tabular}{|c|c| c| c| c| c| c|c|c|}
\hline
 $X$ & degree & $s$ & $r$ & $D$ & $\ell(D)$ & $g$& very ample?& See also: \\\hline
 $\pp^2$ & $9$& $0$ & $1$  & $\cO_{\pp^2}(1)$&3& $0$&yes&\ref{sec:easy}\\\hline
 $Q^{3,1}$ & $8$& $1$ & $0$  & $\cO_{\pp^2}(1,1)$&4& $0$&yes&\ref{sec:easy}\\\hline
 $\pp^2(0,2)$ & $7$& $0$ & $1$  & --- & -& -& -& \ref{cor:non-finite},\ref{rem:non-finite} \\\hline
 $Q^{3,1}(0,2)$ & $6$& $1$ & $0$  & --- & -& -& -&\ref{cor:non-finite} ,\ref{rem:non-finite} \\\hline
 $\pp^2(0,4)$ & $5$& $0$ & $1$  & --- & -& -& -& \ref{cor:non-finite},\ref{rem:non-finite} \\\hline
 $Q^{3,1}(0,4)$ & $4$& $1$ & $0$  & --- & -& -& -&\ref{cor:non-finite},\ref{rem:non-finite}  \\\hline
 $\mathbb{D}_4$ & $4$& $2$ & $0$  & $-K$& 5 & 1&yes&\ref{cor:d4} \\\hline
 $\pp^2(0,6)$ & $3$ & $0$ & $1$ & --- & -& -& -&\ref{cor:non-finite},\ref{rem:non-finite} \\\hline
 $\mathbb{D}_4(1,0)$ & $3$& $1$ & $1$  & $-K$& 4 & 1&yes&\ref{sec:easy} \\\hline
 $\mathbb{D}_4(2,0)^1_1$ & $2$& $0$ & $2$  & $-K$& 3 & 1&no&\ref{sec:easy} \\\hline
 $Q^{3,1}(0,6)$ & $2$& $1$ & $0$  & --- & -& -& -&\ref{cor:non-finite},\ref{rem:non-finite}  \\\hline
 $\mathbb{D}_4(0,2)$ & $2$& $2$ & $0$  & --- & -& -& -&\ref{cor:non-finite},\ref{rem:non-finite}  \\\hline
 $\mathbb{D}_2$ & $2$& $3$ & $0$  & $F-K$ & 6 & 2&yes&\ref{cor:d2}, \ref{geyser}
 \\\hline
  & &  &   & $-F-3K$ & 6 & 2&yes&\ref{cor:d2}, \ref{geyser}
  \\\hline
 $\mathbb{G}_2$ & $2$& $4$ & $0$  & $-2K$ & 7 & 3&yes& \ref{cor:g2} \\\hline
 $\pp^2(0,8)$ & $1$ & $0$ & $1$ & --- & -& -& -&\ref{cor:non-finite},\ref{rem:non-finite} \\\hline
 $\mathbb{D}_4(1,2)$ & $1$& $1$ & $1$  & --- & -& -&-& \ref{cor:non-finite},\ref{rem:non-finite} \\\hline
 $\mathbb{D}_2(1,0)$ & $1$& $2$ & $1$  & $-3K-\tilde{F}+E$ & 5& 2&yes&\ref{cor:d2}, \ref{geyser}
 \\\hline
  & &  &  & $-5K-\tilde{F}-E$ & 5& 2&yes&\ref{cor:d2}, \ref{geyser}
  \\\hline
  & &  &  & $-K+\tilde{F}+E$ & 5& 2&yes&\ref{cor:d2}, \ref{geyser}
  \\\hline
  & &  &  & $-3K+\tilde{F}-E$ & 5& 2&yes&\ref{cor:d2}, \ref{geyser}
  \\\hline
 $\mathbb{G}_2(1,0)$ & $1$& $3$ & $1$  & $-2K+E$ & 6& 3&yes&\ref{cor:g21} \\\hline
  & &  & $1$  & $-4K-E$ & 6& 3&yes&\ref{cor:g21} \\\hline
 $\mathbb{B}_1$ & $1$& $4$ & $1$  & $-3K$ & 7& 4&yes&\ref{prop:b1} \\\hline
\end{tabular}
  \caption{A list of all del Pezzo surfaces $X$ whose real part consists of spheres and real projective planes together with all divisor classes $D$ that satisfy the conditions of \Cref{thm:nec}. The $7$th column keeps track of the genus of the divisor $D$. In the last column, the references are \Cref{sec:easy}, \Cref{cor:non-finite}, \Cref{rem:non-finite}, \Cref{cor:d4}, \Cref{cor:d2}, \Cref{cor:g21} and \Cref{prop:b1}.}\label{tb:list}
\end{table}

In the next section we will prove that in fact all these divisors come from a real-fibered morphism $f:X\to\pp^2$ (and a hyperbolic embedding of $X$ in the very ample cases). The last column of \Cref{tb:list} indicates where a detailed treatment of these divisors can be found. Note that for determining which divisors are very ample we can employ \cite{veryrocco}.
 For now we extract from \Cref{tb:list} the following.

\begin{cor}\label{cor:list}
 Let $X$ be a smooth del Pezzo surface such that $X(\R ) \simeq s S^2 \sqcup r \R \pp^2$. 
 If there is a finite and real-fibered morphism $f:X\to\pp^2$, then the preimage of a generic real line is an $M$-curve, i.e., has genus $r+s-1$. Furthermore, the space of global sections of $f^*\cO_{\pp^2}(1)$ has dimension $s+3$ and if $f^*\cO_{\pp^2}(1)$ is very ample, then $r\in\{0,1\}$.
\end{cor}

\begin{rem}\label{sec:easy}
In some cases it is rather easy to verify that a morphism is real-fibered. For example, it is clear that any real automorphism $\pp^2\to\pp^2$ is real-fibered. Denote by $\mathbb{D}_4(2,0)^1_1$ the blow-up of $\mathbb{D}_{4}$ at two points belonging to different connected components. The anti-canonical map $\mathbb{D}_4(2,0)^1_1\to\pp^2$ is a double cover of $\pp^2$ ramified along a plane quartic curve without real points. This is clearly real-fibered.
Finally, the hypersurfaces $Q^{3,1}$ and $\mathbb{D}_4(1,0)$ in $\pp^3$ are hyperbolic with respect to any point in the interior of the $2$-sphere as each $2$-sphere disconnects $\pp^{3}(\mathbb{R})$, and the intersection of any $\pp^{1}(\mathbb{R})$ with any $\pp^{2}(\mathbb{R})$ is odd (see also \cite[Theorem 5.2]{HV07}).
\end{rem}

\section{Linking lemma and varieties of higher dimension}
\label{sec:linking}
In this section, we give a criterion to determine whether an embedded variety of dimension $k$ is hyperbolic in terms of linking numbers.

\begin{Def} [{\cite[\S 2.5]{prasolov}}]
Let $X,Y$ be disjoint embedded oriented spheres in $S^n$ of dimensions $l$ and $m$ respectively, where $n=l+m+1$.
Consider the fundamental cycles $[X]$ and $[Y]$ as cycles in the integral homology of $S^n$.
There exists a chain $W$ whose boundary is $[X]$.
The \emph{linking number} $\lk (X,Y)$ is defined to be the intersection number of $W$ and $[Y]$. 
\end{Def}

\begin{rem}
For $l$ and $m$ non zero, the linking number does not depend on the choice of $W$.
Indeed, given another chain $W'$ satisfying $\partial W' = [X]$, the chain $W' - W$ is a cycle, and hence the boundary of a chain $e$. 
Then the intersection number between $W' -W$ and $[Y]$ is the intersection number between $\partial e$ and $[Y]$, which is equal to the intersection number between $e$ and $\partial [Y] = 0$, hence it is equal to zero.
If one of $l,m$ is zero (say $m$), then the linking number must be computed as the intersection number of a 1-dimensional chain $W$ with boundary $[Y]$ with the cycle $[X]$ (and it will not depend on the choice of $W$ in that case), as otherwise the intersection number would depend on the choice of chain $W$ with boundary $[X]$.
\end{rem}

We extend the definition of linking numbers to spheres and linear subspaces inside a real projective space, following the idea in \cite[\S 2]{sep} of linking numbers of 1-dimensional spheres and linear subspaces inside a real projective space.

\begin{Def}
Let $K \subset \mathbb{P}^n (\mathbb{R} )$ be an embedded $k$-sphere in $\pp^n (\mathbb{R} $) and $L \subset \mathbb{P}^n (\mathbb{R} )$ be a linear subspace of dimension $n-k-1$.
Let $p : S^n \rightarrow \mathbb{P}^n (\mathbb{R} )$ be an unramified double cover.
The \emph{linking number} $\lk (K,L)$ is defined as the linking number of $K_1 \sqcup K_2$ with $p^{-1} (L)$ in $S^{n}$, where $K_1 \sqcup K_2$ is the preimage of $K$ via $p$ (one of the $K_{i}$ may be empty).
\end{Def}
The following proposition is a generalisation of \cite[Prop.~2.12]{sep} to the case of projective subvarieties of any dimension.
\begin{prop}
\label{LinkLem}
Let $X \subset \mathbb{P}^n$ be a smooth subvariety of dimension $k$ and degree $2s+r$ such that the set of real points $X(\R)$ is homeomorphic to $s S^k \sqcup r \R \pp^k$.
Let $E \subset \mathbb{P}^n$ be a linear subspace of 
dimension $n-k-1$ with $X\cap E = \emptyset$.
Let $X_1,\ldots, X_{r+s}$ be the connected components of $X(\R )$.
Then $X$ is hyperbolic with respect to $E$ if and only if 
\[ \sum\limits_{i=1}^{r+s} |\lk (X_i , E (\R ))| = 2s+r . \]
\end{prop}

\begin{proof}
The variety $X$ is hyperbolic with respect to $E$ if and only if every 
dimension $n-k$ real linear space $L$ that contains $E$ intersects $X$ in $\deg X$ many (distinct) real points.
Let $p : S^n \rightarrow \mathbb{P}^n (\R )$ be an unramified double cover.
For any choice of such an $L \subset \mathbb{P}^n$ containing $E$, the preimage $p^{-1}(E (\R ))$ is a sphere of dimension $n-k-1$ inside $p^{-1}(L (\R ))$ which in turn is a sphere of dimension $n-k$.
Let $W \subset p^{-1} (L (\R ))$ be a hemisphere whose boundary is $p^{-1} (E(\R ))$.
If $X$ is hyperbolic with respect to $E$, then the absolute values of the linking numbers $\lk (X_i , E (\R))$, which are the intersection numbers of the $p^{-1} (X_i)$ with $W$, sum up to $\deg X$.
Conversely, if the intersection number of $W$ with the preimage of $X(\R )$ is $\deg X$, then $L$ has (at least) $\deg X$ many real intersection points with $X$.
\end{proof}

The following proposition is a direct application of \Cref{LinkLem} and is used in \Cref{explicit}. 

\begin{prop}
\label{S2}
Let $X\subset \mathbb{P}^n$, for some $n\geq 3$, be a smooth surface and let $X_0$ be a connected component of $X(\mathbb{R} )$ homeomorphic to $S^2$.
Let $E \subset \mathbb{P}^n$ be a linear subspace of codimension 3 with $E\cap X = \emptyset$.
If $E( \R )$ intersects a 3-dimensional disc $W\subset\pp^n(\R)$, whose boundary is $X_0$, in exactly one point, then $|\lk (X_0 ; E(\R ))| = 2$.
\end{prop}
\begin{proof}
Let $p : S^n \rightarrow \mathbb{P}^n(\R)$ be an unramified double cover, and let $f : S^2 \rightarrow  \mathbb{P}^n(\R)$ be a continuous map. Let $p_* , f_*$ be the corresponding induced homomorphisms between fundamental groups. Since $\pi_1 (S^k) = 0$ for $k\geq 2$, we have that $f_* (\pi_1 (S^2)) \subseteq p_* (\pi_1 (S^n))$ and, by the Lifting Property (\cite[Proposition 1.33]{hatcher2002algebraic}), there exists a lift $\widetilde{f} : S^2 \rightarrow S^n$ of $f$. As homology is functorial (and covariant), it preserves compositions, hence we have $f_* = p_* \circ \widetilde{f}_*$ for
the induced homomorphisms between $\Z_2$-homology groups.
As $H_2 (S^n ; \Z_2) = 0$ for $n\geq 3$, the homomorphism $f_*$ factors through zero, therefore it is the zero homomorphism.
In particular, if $[S^2]\in H_2 (S^2; \Z_2)$ is the fundamental class of $S^2$, then $f_* [S^2] = 0$.
From this, we get that the linking number between $X_0$ and $E( \mathbb{R} )$ must be even.
By assumption, the real part $E(\R )$ intersects a 3-dimensional disc $W \subset \pp^n (\R)$ of boundary $X_0$ in exactly one point, therefore $\lk (X_0 , E(\R )) = \pm 2$.
\end{proof}

We conclude this section by giving a proof of \Cref{thm: alt_linking1}.

\begin{proof}[Proof of \Cref{thm: alt_linking1}]
The variety $C$ is hyperbolic with respect to $E$ if and only if every dimension $n-k$ linear space $L$ that contains $E$ intersects $C$ in $\deg C$ many (distinct) real points.
Let $p : S^{n-1} \rightarrow H (\R )$ be an unramified double cover. For any choice of such a $L \subset H$ containing $E$, the preimage $p^{-1}(E (\R ))$ is a sphere of dimension $n-k-1$ inside $p^{-1}(L (\R ))$ which in turn is a sphere of dimension $n-k$.
Let $W \subset p^{-1} (L (\R ))$ be a hemisphere whose boundary is $p^{-1} (E(\R ))$.
If $C$ is hyperbolic with respect to $E$, then the absolute values of the linking numbers $\lk (C_i , E (\R))$, which are the intersection numbers of the $p^{-1} (C_i)$ with $W$, sum up to $\deg C$ (\Cref{LinkLem}). Now, let $j: S^{n-1} \hookrightarrow S^{n}$, $i: H(\R) \hookrightarrow \pp^{n}(\R)$ and $q: S^{n}\rightarrow \pp^{n}(\R)$ respectively be two inclusions and a double unramified cover. One has that $j \circ p^{-1}= q^{-1}\circ i$.\\
Denote by $\tilde{E}\subset S^{n}$ the image of $p^{-1} (E (\R ))$ via $j$, which is still a sphere of dimension $n-k-1$ in $S^{n}$. The $n-k$ sphere $\tilde{L} : = j(p^{-1} (L(\R ))$ has $\tilde{E}$ as an equator, and $\tilde{W} := j(W)$ is one of its two hemispheres of boundary $\tilde{E}$. It follows that $q^{-1}(X_{i})$ has to intersect $\tilde{W}$ in at least $|\lk (C_i , E (\R ))|$ number of points. 
Therefore, 
\begin{equation}
\label{eqn: link_dimk}
\sum\limits_{i=1}^{r+s} |\lk (X_i , i( E(\R )))| \geq 2s+r,
\end{equation}
as $\tilde{E} = j (p^{-1} (E(\R )) = q^{-1} (i (E(\R ))$.
Moreover the inequality (\ref{eqn: link_dimk}) is an equality, since
the sum on the left-hand side is at most $\deg (X) = 2s+r$. Thus by Proposition \ref{LinkLem}, the variety $X$ is hyperbolic with respect to $i(E) = E$. 

The converse can be seen by taking the restriction of the set of linear spaces of dimension $n-k-1$ through $E$ to the hyperplane $H$. Every such linear space intersects $C= X\cap H$ in $2s+r$ real points. The hyperbolicity of $X$ with respect to $E$ implies that $C$ is hyperbolic with respect to $E$.
\end{proof}

\section{Hyperbolic del Pezzo surfaces}
\label{sec:hyppezzo}

In this section we will prove Theorems \ref{thm : RealFibDelPezzo}, \ref{thm : HypDelPezzo} and \ref{thm:main}.
 
\begin{lem}\label{lem:notrv}
 Let $X\subset\pp^n$ be a smooth surface such that $X(\R )$ is homeomorphic to $s S^2 \sqcup \R \pp^2$ 
 with $\deg(X)=2s+1$. The connected component that is homeomorphic to a real projective plane realises the nontrivial homology class in $H_2(\mathbb{P}^{n}(\R) ; \Z_2)$.
\end{lem}

\begin{proof}
 Since $X$ has odd degree, its real part $X(\R)$ realises the nontrivial homology class in $H_2(\mathbb{P}^{n}(\R) ; \Z_2)$.  Every sphere embedded to $\mathbb{P}^{n}(\R)$ is homologous to zero (see proof of \Cref{S2}). Thus the remaining connected component must realise the nontrivial class.
\end{proof}

\begin{lem}\label{lem:alt2}
 Let $X\subset\pp^n$ be a smooth nondegenerate surface such that $X(\R )$ is homeomorphic to $s S^2 \sqcup r \R \pp^2$ 
 with $\deg(X)=2s+r$, $n=s+2$ and $r \in \{0,1\}$. Assume that the sectional genus of $X$ is $s+r-1$. There is a hyperplane $H\subset\pp^n$ such that $C=X\cap H$ is a smooth and nondegenerate $M$-curve with the property that each connected component of $X(\R)$ contains exactly one connected component of $C(\R)$.
\end{lem}
\begin{proof}
 For any choice of one point $p_i$ on each connected component of $X(\R)$ that is homeomorphic to a sphere, there is a hyperplane of $\pp^n$ that contains these points since $n>s$. If we choose these points general enough, then by Bertini's lemma there is such a hyperplane $H$ whose intersection with $X$ is smooth and nondegenerate in $H$. Let $C$ be $X\cap H$. By construction and \Cref{lem:notrv}, each connected component of $X(\R)$ contains at least one connected component of $C(\R)$. But since the genus of $C$ is $s+r-1$, it must be exactly one connected component of $C(\R)$ on each connected component of $X(\R)$. 
\end{proof}

\begin{lem}\label{lem:alt3}
 Let $C\subset\pp^{s+1}$ be a smooth nondegenerate $M$-curve of genus $g=s+r-1$ and degree $2s+r$ such that $r$ components of $C(\R)$ realise the nontrivial homology class in 
$H_1(\mathbb{P}^{s+1}(\R) ; \Z_2)$. Then $C$ is hyperbolic.
\end{lem}
\begin{proof}
 First we note that $C$ has $s+r$ connected components. A general enough hyperplane $H$ that intersects each of the $s$ connected components $C_1,\ldots,C_s$ of $C(\R)$ that realise the trivial homology class will intersect $C$ in $2s+r$ distinct real points. Let $D$ be the divisor corresponding to this hyperplane section. It is of the form $$D=D_0+\sum_{i=1}^sP_i$$ for some effective divisor $D_0$ and  points $P_i\in C_i$ that are not in the support of $D_0$. Note that $D_0$ is the sum of one point from each connected component of $C(\R)$. The divisor $D_k=D_0+\sum_{i=1}^kP_i$ is nonspecial by \cite[Thm.~2.5]{nsd} for all $k=1,\ldots,s$. The complete linear system $|D_0|$ has dimension $2$ and the corresponding morphism $C\to\pp^1$ is real-fibered by \cite[Prop.~4.1]{gabard}. Therefore, an iterated application of \cite[Prop.~3.2]{sep} shows together with \cite[Lem.~2.10]{sep} that the embedding of $C$ via the complete linear system $|D|$ is hyperbolic. But since $D$ is nonspecial, Riemann--Roch shows that this is exactly our embedding $C\subset\pp^{s+1}$ we started with.
\end{proof}
\begin{thm}\label{thm:suff}
 Let $X\subset\pp^n$ be a smooth nondegenerate surface such that $X(\R)$ is homeomorphic to $s S^2 \sqcup r \R \pp^2$ with $\deg(X)=2s+r$, $r\in\{0,1\}$ and $n=s+2$. Assume that the genus of a hyperplane section on $X$ is $s+r-1$. Then $X$ is hyperbolic.
\end{thm}

\begin{proof}
 By \Cref{lem:alt2} there is a hyperplane $H\subset\pp^n$ such that $C=X\cap H$ is a smooth and nondegenerate $M$-curve with the property that each connected component of $X(\R)$ contains exactly one connected component of $C(\R)$. Note that by \Cref{lem:notrv} exactly $r$ connected components of $C(\R)$ realise the nontrivial homology class. This curve $C$ is hyperbolic by \Cref{lem:alt3}. Thus $X$ is hyperbolic by \Cref{thm: alt_linking1}.
\end{proof}

\begin{prop}
\label{cor:d2}
 All the divisors listed in \Cref{tb:list} correspond to a real-fibered morphism. In addition, those divisors which are very ample correspond to a hyperbolic embedding.
\end{prop}

\begin{proof}
 This follows from \Cref{thm:suff} together with \Cref{cor:list} and \Cref{sec:easy}.
\end{proof}

\begin{proof}[Proof of Theorem \ref{thm : RealFibDelPezzo}]
All del Pezzo surfaces admitting a real-fibered morphism (i.e. those given by \Cref{cor:d2}) satisfy one of the conditions $(1),(2),(3)$.  
Namely, the surfaces $\pp^2 , \mathbb{G}_2$ and $\mathbb{B}_1$ have real Picard rank 1, the surfaces $\mathbb{D}_4$ and $\mathbb{D}_2$ are conic bundles of real Picard rank 2, and the other surfaces listed are blow-ups of one of the previous surfaces at one or two real points.  
\end{proof}

\begin{proof}[Proof of \Cref{cor:non-finite}]
Excluding the cases already treated in \Cref{thm : RealFibDelPezzo} and those where $X(\R)$ is not homeomorphic to $sS^{2} \sqcup r \mathbb{R}P^{2}$, we get that each remaining del Pezzo surface $X$  is a blow-up at pairs of complex conjugate points of another del Pezzo surface which admits finite real-fibered morphisms. Therefore $X$ does admit real-fibered morphisms to $\pp^2$.
\end{proof}

\begin{proof}[Proof of Theorem \ref{thm : HypDelPezzo}]
All del Pezzo surfaces admitting a hyperbolic embedding (i.e. those given by \Cref{cor:d2} with the very ampleness condition) satisfy one of the conditions $(1),(2),(3)$. 
Those satisfying $(1)$ or $(2)$ are the same as for Theorem \ref{thm : RealFibDelPezzo}, and all other surfaces listed satisfy $(3)$, except for $\mathbb{D}_4 (2,0)^1_1$, whose divisor was the only one not satisfying very ampleness.  
\end{proof}

\begin{proof}[Proof of Theorem \ref{thm:main}]
 First assume that we have $(1)$---$(3)$ from \Cref{thm:main}. Then the divisor $D$ given by a generic hyperplane section satisfies $(1)$---$(5)$ from \Cref{thm:nec}. 
The conditions $(1)$, $(2)$ and $(5)$ are clear and $(3)$, $(4)$ follow from our assumption on the genus of a hyperplane section. Indeed, the Adjunction Formula implies that $r=D.K+4$. Thus by \Cref{tb:list} we also have $r\in\{0,1\}$ and $n=s+2$. Now \Cref{thm:suff} implies that $X$ is hyperbolic. 
 Conversely, if $X$ is hyperbolic, then $(1)$---$(3)$ follow from \Cref{cor:list}.
\end{proof}
\begin{rem}
\label{rem:non-finite} 
Let $X$ be either $Q^{3,1}(0,2h)$ with $1 \leq h \leq 3$, or $\mathbb{P}^{2}(0,2j)$ with $1 \leq j \leq 4$, or $\mathbb{D}_{4}(0,2)$, or $\mathbb{D}_{4}(1,2)$. Each $X$ admits a real-fibered morphism (\Cref{cor:non-finite}) but no finite real-fibered morphism (\Cref{thm : RealFibDelPezzo}). Therefore, each such $X$ can only admit non-finite real-fibered morphisms.
\end{rem}
\begin{rem}
\label{geyser}
Let us consider the surface $\mathbb{D}_{2}$ and the divisors $D_{1}=F-K$ and $D_{2}=-F-3K$. Thanks to \Cref{thm:suff}, both divisors correspond to hyperbolic embeddings in $\pp^{5}$. Observe that $D_{i}$ is obtained by applying the Geiser involution to $D_j$, where $\{i,j\}=\{1,2\}$ (see \cite[\S 4, Example 3]{russo} for a description of Geiser involution). A similar approach works for the surface $\mathbb{D}_{2}(1,0)$ and the divisors $D_1=-3K-\tilde{F}+E$, $D_2=-3K+\tilde{F}-E$, $D_3=-5K-\tilde{F}-E$ and $D_4=-K+\tilde{F}+E$, one has that all divisors correspond to hyperbolic embeddings in $\pp^{4}$. Moreover, the Bertini involution sends $D_{j}$ to $D_{j+1}$, for $j=1,3$ (see \cite[\S 5 , Example 4]{russo}).
\end{rem}

\section{Explicit constructions and examples}
\label{explicit}
Here we construct for most of the embeddings in \Cref{tb:list} explicit linear subspaces with respect to which the del Pezzo surfaces under consideration are hyperbolic.

\begin{ex}
\label{prop:b1}
Consider the embedding $h: \mathbb{B}_1 \rightarrow \mathbb{P}^{6}$ associated to $|-3K|$.
We will explicitly construct linear subspaces $E\subset H\subset \pp^6$ of dimension $3$ and $5$ to which we can apply Theorem \ref{thm: alt_linking1}. 
The anti-bicanonical map $\phi$ of $\mathbb{B}_1$ is the double cover of the quadratic cone $Q$ in $\mathbb{P}^{3}$ ramified along the vertex $V$ of $Q$ and a real non-singular cubic section $S \subset Q$ disjoint from $V$ that is an $M$-curve of genus four. Via the map $\phi$, we want to construct three smooth curves $C_{1}$, $C_{2}$ and $C_{3}$ on $X$ that are linearly equivalent to $-3K$. Let us pick $C_{3}$ as $\phi^{-1}(S)$ (set-theoretical preimage). We observe that each connected component of $\mathbb{B}_1(\R)$ contains exactly one connected component of $C_3(\R)$. We construct $C_1$ and $C_2$ as follows. Choose a point $p_{i}$ on each connected component of $Q(\mathbb{R}) \setminus S(\mathbb{R})$ homeomorphic to a disk for $i=1,2,3,4$. Pick two curves $C_{ijk}$ and $C_{jkt}$ on $Q$ as the intersection of $Q$ with the hyperplane passing through $p_{i},p_{j},p_{k}$ and $p_{j},p_{k},p_{t}$ respectively, where $\{i,j,k,t\}=\{1,2,3,4\}$. Moreover, pick the two generatrices $L_{t}$ and $L_{i}$ of $Q$ passing through $p_{t}$ and $p_{i}$ respectively (see Fig. \ref{fig: DP1}). One can perturb the union of $C_{ijk}$ and $L_{t}$ resp. the union of $C_{jkt}$ and $L_{i}$ to a smooth curve $X_{1}$ resp. $X_{2}$ (see \cite[Section 4]{Manz20} for details) such that $S(\mathbb{R})\cap X_{i}(\mathbb{R})$ consists of nine distinct real points and $X_{i}(\mathbb{R})$ intersects each connected component of $S(\mathbb{R})$, for $i=1,2$. Then we let $C_i=\phi^{-1}(X_i)$ for $i=1,2$. We further choose hyperplanes $H_i\subset\mathbb{P}^6$ such that $C_i=X\cap H_i$ for $i=1,2,3$. The linear subspace $E=H_1\cap H_2\cap H_3$ has dimension $3$ and the divisors $H_i\cap C_3$ on $C_3$ for $i=1,2$ interlace on $C_3$ in the sense of \cite[\S 2.1]{sep}. Thus \cite[Lemma 2.1]{sep} shows that $C_3$ is hyperbolic with respect to $E$. Therefore, by Theorem \ref{thm: alt_linking1} applied to $E$ and $H=H_3$ we find that $\mathbb{B}_1$ is hyperbolic with respect to $E$ as well.
\begin{figure}
\begin{picture}(100,50)
\put(-95,3){\includegraphics[width=0.90\textwidth]{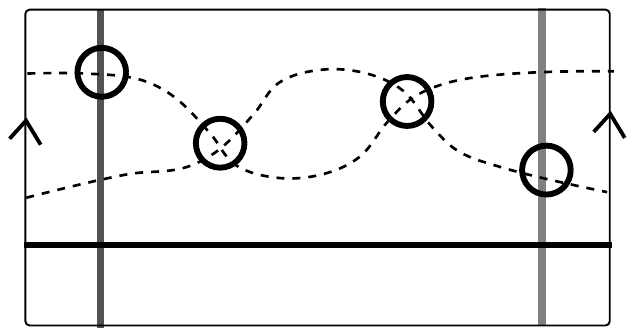}} 
\end{picture}
\caption{The quadrangle, whose vertical sides are identified accordingly with the arrows and horizontal sides represent the vertex $V$, is $Q(\R)$. The cubic section $S(\R)$ is in thick black, while $C_{ijk}(\R)$ and  $C_{jkt}(\R)$ in dashed. The generatrices $L_{i}(\R)$ and $L_{t}(\R)$ are in grey.}
\label{fig: DP1}
\end{figure}
\end{ex}

\begin{rem}
\label{rem: alternating}
Let $X$ be a real non-singular del Pezzo surface and $D$ a real very ample divisor on $X$. An analogue construction to \Cref{prop:b1} can be applied to the embedding associated to $|D|$, where $X=\mathbb{G}_{2}$ and $D=-2K$. 
\end{rem}
 Now, we study in more details the cases of real del Pezzo surfaces obtained as a double cover of some real surfaces ramified along a real curve.
In particular, the real part of these surfaces consist of spheres. The double cover assumption allows us to talk about the ``interior" of these spheres, enabling us to choose a suitable linear space $E$ in order to apply Proposition \ref{LinkLem}. 

\begin{prop}\label{Double}
 Let $Y\subset\pp^n$ be a smooth surface of degree $n-2$ contained in some hyperplane $H\subset\pp^n$. Let $p\in\pp^n$ be a real point that is not in $H$ and let $\tilde{Y}\subset\pp^n$ be the cone over $Y$ with apex $p$. Finally, let $X\subset\tilde{Y}$ be a smooth surface with $X(\R)\neq\emptyset$ that does not contain $p$ such that the projection $\pi_p: X\to Y$ is a double cover branched along the intersection $C$ of $Y$ with a quadratic hypersurface. 
 \begin{enumerate}
  \item If $C(\R)=\emptyset$ and $Y$ is hyperbolic, then $X$ is hyperbolic.
  \item If $C(\R)\neq\emptyset$ and $X(\R)$ is homeomorphic to the disjoint union of $n-2$ spheres, then $X$ is hyperbolic. Furthermore, for any $v\in X(\R)$ the embedding of $\Bl_v X$ to $\mathbb{P}^{n-1}$ obtained by projecting $X$ from $v$ is also hyperbolic.
 \end{enumerate}
\end{prop}

\begin{proof}
 Without loss of generality, we can assume that $H=\cV(x_0)$ and $p=[1:0:\cdots:0]$. Then there is a quadratic polynomial $f\in\R[x_1,\ldots,x_n]$ such that $X$ is the intersection of $\tilde{Y}$ with $\cV(x_0^2-f)$. In case $(1)$ the quadratic polynomial $f$ is strictly positive on the real part of $Y$ and the map $\pi_p$ is real-fibered. Thus if $Y$ is hyperbolic with respect to a linear subspace $E\subset H$, then $X$ is hyperbolic with respect to the subspace spanned by $E$ and $p$. Now assume that we are in case $(2)$. Consider the set $W\subset\tilde{Y}(\R)$ of all points $x$ with $x_0^2\leq f(x)$. This set has $n-2$ connected components $W_1,\ldots,W_{n-2}$ and the boundary of each $W_i$ is a connected component $X_i$ of $X(\R)$. Let $E$ be a linear space of dimension $n-3$ that intersects each $W_i\setminus X_i$ in (at least) one point $p_i$. Then since $\deg(\tilde{Y})=n-2$ there are no further intersection point of $E$ and $\tilde{Y}$ (note that this implies that $E$ is spanned by the $p_i$ and that $E\cap X=\emptyset$). Thus $E$ intersects each $W_i$ in exactly one point, so by \Cref{S2} we have $|\lk (X_i ; E(\R ))| = 2$ for $i=1,\ldots,n-2$. Therefore, by \Cref{LinkLem} the variety $X$ is hyperbolic with respect to $E$. For the additional statement assume without loss of generality that $v\in X_1$ and take $v_j\in W_1\setminus X_1$ a sequence of points that converges to $v$. Let $E_j$ be the linear subspace of dimension $n-3$ that is spanned by $v_j$ and $p_2,\ldots, p_{n-2}$. This sequence converges to the linear subspace $\tilde{E}$ that is spanned by $v$ and $p_2,\ldots, p_{n-2}$. Thus, since $X$ is hyperbolic with respect to each $E_j$, every linear subspace $E'$ of dimension $n-2$ that contains $\tilde{E}$ intersects $X$ only in real points. This implies that the image of $X$ under the projection from $v$ is hyperbolic with respect to the image of $\tilde{E}$ under this projection.
\end{proof}

\begin{ex}\label{cor:d4}
 The anticanonical divisor on $X= \mathbb{D}_4$ gives an embedding into $\mathbb{P}^4 = \mathbb{P}^{n+2}$ where $n=2$ is the number of spheres in $X(\R)$. The image is cut out by a pencil of quadrics. The corresponding complex pencil contains (counted with multiplicity) five singular quadrics. Thus the image of $\mathbb{D}_4$ is contained in at least one real singular quadric $Q\subset\pp^4$. The projection from the vertex of $Q$ realises $\mathbb{D}_4$ as a double cover of a quadratic hypersurface $Y\subset\pp^3$ ramified along a smooth curve $C$ of bidegree $(2,2)$. If $C(\R)=\emptyset$, then $Y=Q^{3,1}$ and we can apply part $(1)$ of \Cref{Double} to obtain a plane $E$ with respect to which the image of $ \mathbb{D}_4$ is hyperbolic. Otherwise, we can apply part $(2)$ of \Cref{Double}.
\end{ex}

\begin{ex}\label{cor:g2}
 Consider the embedding $\mathbb{G}_2 \rightarrow \mathbb{P}^6$ associated to $|-2K|$.
 We have $l(-2K) = 7$, hence $|-2K|$ embeds $\G_2$ into $\mathbb{P}^6 = \mathbb{P}^{n+2}$ for $n$ the number of spheres in the real part of $\mathbb{G}_2$. The canonical map $\G_2\to\pp^2$ is a double cover of $\mathbb{P}^2$ ramified along a smooth quartic curve with four connected components in its real part. Therefore, we can apply part $(2)$ of \Cref{Double} to $Y\subset\pp^5$ being the image of $\pp^2$ under second Veronese map.
\end{ex}

\begin{ex}\label{cor:g21}
 Now consider the embeddings $\mathbb{G}_2(1,0) \rightarrow \mathbb{P}^5$ associated to $|D_i|$ for $i=1,2$ with $D_1=-2K+E$ and $D_2=-4K-E$.
 By applying part $(2)$ of \Cref{Double} to the situation considered in \Cref{cor:g2}, we obtain a hyperbolic embedding of $\mathbb{G}_2(1,0)$ to $\pp^5$. 
 This correspond to the divisors $D_1$. The other divisor is obtained by applying the Bertini involution $\tau$ on $D_1$ ( \cite[\S 5 , Example 4]{russo}).
\end{ex}

\section{Conic bundles}
\label{conic_bundles}
In this section $\pi: X\to\pp^1$ will denote a geometrically irreducible smooth minimal conic bundle, i.e. each fiber of $\pi$ is isomorphic to a plane conic and the real Picard rank of $X$ is $2$. Assume that $X(\R)$ consists of $s$ spheres. Let $F$ denote a fiber of $\pi$. Then $\textrm{Pic}(X)$ is generated by $-K$ and $F$. We clearly have $F.F=0$. We further have $K.K=8-2s$ and $F.K=-2$, see \cite[p.~5]{calabi}. Applying the necessary criteria for the existence of a real-fibered morphism $X\to\pp^2$ from \Cref{thm:necsurf} to this situation, gives the following.

\begin{cor}\label{thm:necbund}
 Let $f: X\to\pp^2$ be a finite real-fibered morphism and $D=aF-bK$, $a,b\in\Z$, the corresponding ample divisor class. Then we have the following:
 \begin{enumerate}
  \item $b\geq1$;
  \item $a\geq-1$;
  \item $s=b\cdot((4-s)\cdot b+2a)$;
  \item $a+(4-s)\cdot b\leq 2$;
  \item $a\equiv s\cdot b \mod 2$;
  \item $2a>b(s-4)$;
 \end{enumerate}
\end{cor}

\begin{proof}
 Since $F$ is effective and $D$ ample we must have $2b=F.D>0$. This shows $(1)$. By \Cref{thm:necsurf}(2) we have $$2s=D.D=4ab+b^2(8-2s)$$which shows $(3)$. By \Cref{thm:necsurf}(3) we have $$0\leq D.K+4=-2a-b\cdot(8-2s)+4$$which shows $(4)$. Part $(5)$ follows from \Cref{thm:necsurf}(4). Finally, part $(4)$ implies $$2a+(4-s)\cdot b\leq 2+a.$$ Multiplying this with $b$ and using part $(3)$ we obtain $$s=(2a+(4-s)\cdot b)\cdot b\leq(2+a)\cdot b.$$ Since $b,s>0$, we obtain $2+a>0$. For part $(6)$ we observe that since $D$ is ample we must have $$0<D.D=4ab+b^2(8-2s).$$ Since $b>0$, this shows $0<2a+b(4-s)$.
\end{proof}

\begin{lem}
\label{lem:RR_conic}
Let $D=aF-bK$ be an ample divisor class satisfying $(1)-(6)$ of \Cref{thm:necbund}. Then $\ell(D)\geq s+3$.
\end{lem}

\begin{proof}
 By Riemann--Roch and since $X$ is rational we have $$\ell(D)+\ell(K-D)\geq\frac{1}{2}D.(D-K)+1=s+3 .$$ Thus it suffices to show that $\ell(K-D)=0$. We note that $F$ is nef as the pullback of a nef divisor \cite[Exp.~1.4.4]{nefl}. But since we have $(K-D).F=-2(1+b)$, the divisor $K-D$ cannot be effective.
\end{proof}

This gives us a good candidate for a linear system giving rise to a hyperbolic embedding. Namely, for any fixed $s$, the divisor $D=(s-2)F-K$ satisfies all of the necessary conditions in \Cref{thm:necbund} and we have the following.

\begin{cor}
\label{cor: conic_very_ample}
 If $D=(s-2)F-K$ is very ample, then it gives rise to a hyperbolic embedding of $X$.
\end{cor}

\begin{proof}
 By the Adjunction Formula we compute the sectional genus of $D$ as $$\frac{1}{2}D.(D+K)+1=\frac{1}{2}((s-2)F-K).(s-2)F+1=s-1.$$ 
 Thus by \Cref{lem:RR_conic} we can apply \Cref{thm:suff}.
\end{proof}

\begin{rem}
The divisor $D=(s-2)F-K$ from \Cref{cor: conic_very_ample} is the only divisor for which equality holds in \Cref{thm:necsurf}(3).
\end{rem}

\begin{rem}
Observe that $\mathbb{D}_{4}$ and $\mathbb{D}_{2}$ are minimal conic bundles with $s=2$ and $s=3$ respectively. The only divisors of these surfaces corresponding to real-fibered morphisms are of the form $(s-2)F-K$ (\Cref{tb:list}). Note that in the case of $\D_2$ we have two different divisors which is due to the fact that $\D_2$ can be equipped with two different structures of a conic bundle. In both cases, these divisors are very ample.
\end{rem}

We focus on the divisor class $D=(s-2)F-K$ on a minimal conic bundle $X$. We first observe the existence of a particular rank $3$ bundle $\mathcal{E}$ for any given $X$ by \Cref{calc}. 
Then we look at some minimal conic bundles for which $D$ is always very ample.
\begin{prop}
\label{calc}
There is a vector bundle $\mathcal{E}$ of rank 3 on $\pp^1$ with first Chern class $s$ and a section $t$ of $\mathcal{O}_{\mathbb{P}(\mathcal{E})} (2)$ such that $X$ is the zero set of $t$ and $\pi$ is the restriction of the natural projection $f:\mathbb{P} (\mathcal{E})\rightarrow \pp^1$ to $X$.
Furthermore, the restriction $\mathcal{O}_{\mathbb{P}(\mathcal{E})} (1)|_X$ corresponds to the class $(s-2)F-K$.
\end{prop}

\begin{proof}
By \cite[Ex. 3.13.4]{nearlyrat}, there is a rank 3 vector bundle $\mathcal{E} '$ on $\pp^1$ and an embedding $X\rightarrow \pp^1$ where $X$ is realised as a family of conics in the projective plane fibers of $\pp (\mathcal{E} ')$. From  \cite[\S 9.3]{3264}, the Chow ring of $\pp (\mathcal{E} ')$ is given by $\Z [H,E] / (E^2 , H^3 - cEH^2)$, where $H$ is the class of the line bundle $\mathcal{O}_{\pp (\mathcal{E} ')} (1)$, the class $E$ is a fiber of the map $\pp (\mathcal{E} ') \rightarrow \pp^1$ and $c$ is the first Chern class of $\mathcal{E} '$.
The class of a point is $EH^2$.
Because $X$ has codimension $1$ in $\pp(\cE')$, its class of $X$ in the Chow ring of $\pp (\mathcal{E} ')$ must be of the form $aE+a'H$ for some integers $a,a'$. Because the intersection of $X$ with a fiber $E$ is a plane conic, we must have $a'=2$.
The canonical class on $\pp (\mathcal{E} ')$ is $(c-2)E-3H$.
Thus by the Adjunction Formula \cite[Prop.~1.33]{3264} the canonical class of $X$ can be obtained by intersecting $((a+c-2)E-H)$ with $X$. This implies that we have 
\[ K.K  = (( a+c -2)E-H)^2.(2H+aE)=8-3a-2c. \]
On the other hand we know that $K.K = 8-2s$ which implies that $s=3b+c$ where $b$ is an integer satisfying $2b=a$.
Consider the vector bundle $\mathcal{E} = \mathcal{E}'(b)$ on $\pp^1$.
The first Chern class of $\mathcal{E}$ is $c+3b=s$ \cite[Prop.~5.17]{3264} and $\pp (\mathcal{E} )$ is isomorphic to $\pp (\mathcal{E} ')$ as scheme over $\pp^1$ \cite[Cor.~9.5]{3264}.
Finally, we have that $\mathcal{O}_{\mathbb{P}(\mathcal{E})} (1) = \mathcal{O}_{\mathbb{P}(\mathcal{E} ')} (1) \otimes f^* \mathcal{O}_{\pp^1}(b)$ \cite[Cor.~9.5]{3264}, i.e., the zero set of a nonzero section in $\mathcal{O}_{\mathbb{P}(\mathcal{E})} (1)$ has the class $H+bE$.
This shows that the class corresponding to $\mathcal{O}_{\mathbb{P}(\mathcal{E})} (2)$ is $2H+2bE$, the class of $X$.
In order to prove the additional statement we can, after replacing $\mathcal{E}'$ by $\mathcal{E}$, assume without loss of generality that $c=s$ and $a=0$.
Let $xF+yK$ be the class of $\mathcal{O}_{\mathbb{P}(\mathcal{E})} (1)|_X$.
On one hand, we have $(xF+yK).F = -2y$ in the Chow ring of $X$.
We can compute the same number in the Chow ring of $\pp (\mathcal{E})$ as 
\[ 2H.H.E = 2 \]
which shows that $y=-1$.
Similarly, we have in the Chow ring of $X$ that $(xF-K).K = -2x-8+2s$.
This can be computed in the Chow ring of $\pp (\mathcal{E})$ as 
\[ 2H.H.((s-2)E-H) = -4   \]
which implies $x=s-2$.
\end{proof}

\begin{rem}
Conversely, let $\mathcal{E}$ be a vector bundle of rank 3 on $\pp^1$ with first Chern class $s$ and $X$ the smooth zero set of a section of $\mathcal{O}_{\pp (\mathcal{E})}(2)$.
Then clearly the restriction of the natural projection $\pp (\mathcal{E}) \rightarrow \pp^1$ to $X$ gives $X$ the structure of a conic bundle.
A direct computation as in the proof of Proposition \ref{calc} shows that $X$ has $2s$ singular fibers and that the restriction of $\cO_{\pp(\cE)}(1)$ to $X$ corresponds to the divisor class $(s-2)F-K$.
\end{rem}

Let $\mathcal{E}$ be a vector bundle as in \Cref{calc}. By Grothendieck's splitting theorem (\cite[Theorem 6.29]{3264}) we have $\mathcal{E}=\mathcal{O}_{\pp^1}(a_1)\oplus\mathcal{O}_{\pp^1}(a_2)\oplus\mathcal{O}_{\pp^1}(a_3)$ where the integers $a_i$ sum up to $s$. If each $a_i>0$, then $\mathcal{O}_{\pp(\mathcal{E})}(1)$ is very ample \cite[Proof of Cor.~9.9]{3264}, and therefore $D$ is also very ample on $X$. It follows that one can construct explicit examples in which $D$ is very ample and embeds $X$ in some projective space as a hyperbolic variety, see \Cref{ex2_final}.

\begin{ex}[Very ample]
\label{ex2_final}
Consider the zero set $X$ of the bihomogeneous polynomial 
\[ G = uvx_0^2 + (u^2-v^2)x_1^2 + (u^2-4v^2)x_2^2   \]
inside $\pp^1 \times \pp^2$.
Clearly $X$ is a conic bundle with $2s = 6$ singular fibers.
Letting $\mathcal{E}=\mathcal{O}_{\pp^1}(1)^3$ we have $\pp (\mathcal{E}) = \pp^1 \times \pp^2$ and $G$ is a global section of $\mathcal{O}_{\pp (\mathcal{E})} (2)$.
The Segre embedding of $\pp^1 \times \pp^2$ to $\pp^5$ is the embedding associated to $\mathcal{O}_{\pp (\mathcal{E})} (1)$.
Via this embedding $X$ is hyperbolic.

\end{ex}
Here is an example in which the very ampleness fails.
\begin{ex}[Not very ample]
\label{ex1_final}
Consider the vector bundle $\mathcal{E} = \mathcal{O}_{\pp^1}^2 \oplus \mathcal{O}_{\pp^1} (2)$ and let $\pp (\mathcal{E}) = \Proj (\Sym \mathcal{E})$ the associated projective plane bundle over $\pp^1$.
We define $X\subset \pp (\mathcal{E})$ to be the Zariski closure of the zero set of 
\[ (t^3 - t)x_0^2 + x_1^2 + x_2^2   \]
in a chart $\mathbb{A}^1 \times \mathbb{P}^2 \subset \pp (\mathcal{E}).$
One can check that $X$ is a smooth conic bundle with singular fibers at $t=-1,0,1,\infty$.
It is a desingularisation of the intersection of 
\[ S(0,0,2) = \mathcal{V}(y_0 y_2 -y_1^2) \subset \pp^4   \]
with the quadratic hypersurface defined by 
\[ y_3^2 + y_4^2 - y_0 y_1 + y_1 y_2 = 0.  \]
For example using the package \texttt{Divisor} \cite{DivisorSource, DivisorArticle} for the computer algebra system \texttt{Macaulay2} \cite{M2} one can show that in this case $-K$ is not very ample.
\end{ex}

If not all of the $a_i$ are positive, we can still get a hyperbolic embedding under some conditions.

\begin{prop}
Let $\mathcal{E} = \mathcal{O}_{\pp^1} \oplus \mathcal{O}_{\pp^1}(a_1) \oplus \mathcal{O}_{\pp^1}(a_2)$ be the vector bundle from Proposition \ref{calc} with $0 < a_1 \leq a_2$ (and $a_1 + a_2 = s$ by Proposition \ref{calc}).
Then we have an hyperbolic embedding $X \rightarrow \pp^{s+2}$ if the image of $X$ does not contain the vertex of $S(0,a_1,a_2)$. 
\end{prop}

\begin{proof}
The morphism induced by $\mathcal{O}_{\pp (\mathcal{E})} (1)$ is an immersion \[ \varphi : \pp (\mathcal{E}) \rightarrow \pp^{s+2} = \pp^{a_1 +a_2 + 2} . \] 
The image of $\pp (\mathcal{E})$ via $\varphi$ is the rational normal scroll $S (0,a_1 ,a_2)$, which is a cone over the rational normal scroll $S(a_1 , a_2)$ \cite[\S 9.1]{3264}.
The immersion $\varphi$ is an embedding except for the vertex of the cone.
Since $X$ is given as the zero set of a section of $\mathcal{O}_{\pp (\mathcal{E})} (2)$, the restriction of $\varphi$ to $X$ is an immersion, which is an embedding if $\varphi (X)$ does not contain the vertex of $S(0, a_1 , a_2)$.
Since the restriction $\mathcal{O}_{\pp (\mathcal{E})} (1) |_X$ is associated to the divisor $D = (s-2)F-K$ on $X$ by Proposition \ref{calc}, and by Corollary \ref{cor: conic_very_ample}, the embedding $\varphi |_X$ is hyperbolic. 
\end{proof}

We conclude this section with the following question on the topology of hyperbolic surfaces.

\begin{qu}
 For which pairs $(s,r)$ does a smooth irreducible hyperbolic surface $X\subset\pp^n$ exist such that $X(\R)$ is homeomorphic to the disjoint union of $s$ spheres and $r$ real projective planes?
\end{qu}

From \cite[Theorem 5.2]{HV07} it is known that for arbitrary $s\in\Z_{\geq0}$ and $r\in\{0,1\}$ there is a smooth irreducible hyperbolic hypersurface $X\subset\pp^3$ such that $X(\R)$ is homeomorphic to the disjoint union of $s$ spheres and $r$ real projective planes. Moreover, when $s\in\{0,1\}$, then we must have $r\in\{0,1\}$ as well. This follows from \cite[Lemmas 2.16 and 2.17]{sep} by intersecting $X$ with a generic hyperplane containing the space of hyperbolicity. We want to show that for $s\geq3$, we can have arbitrary $r\in\Z_{\geq0}$.

\begin{prop}
 Let $s\geq3$ and $r\geq0$. There exists a smooth irreducible hyperbolic surface $X\subset\pp^n$ such that $X(\R)$ is homeomorphic to the disjoint union of $s$ spheres and $r$ real projective planes.
\end{prop}

\begin{proof}
 Let $a_1,a_2,a_3>0$ such that $s=a_1+a_2+a_3$ and let $\cE=\mathcal{O}_{\pp^1}(a_1)\oplus\mathcal{O}_{\pp^1}(a_2)\oplus\mathcal{O}_{\pp^1}(a_3)$. We denote by $H$ the class of the line bundle $\mathcal{O}_{\pp (\mathcal{E})} (1)$ and by $E$ the class of a fiber of the map $\pp (\mathcal{E}) \rightarrow \pp^1$. For $i=1,2,3$, let $p_i\in\R[u,v]_{2a_i}$ be a binary form of degree $2a_i$ which has only simple, real zeros and assume that the $p_i$ are pairwise coprime. Then $$t=p_1x_1^2+p_2x_2^2+p_3x_3^2$$ is a global section of $\cO_{\pp(\cE)}(2)$. Here the $x_i$'s are the coordinates on the projective plane fibers. After a small perturbation if necessary, the zero set $X$ of $t$ is smooth and irreducible by Bertini's Theorem \cite[Thm.~6.10]{bertini}. It is thus a minimal conic bundle which has by construction $2s$ singular fibers. 
 We also note that $X(\R)$ is homeomorphic to the disjoint union of $s$ spheres. Now let $X'$ be the union of $X$ with $r$ different fibers of $\pp (\mathcal{E}) \rightarrow \pp^1$ all of whose real parts are disjoint from $X(\R)$. This guarantees that $X'(\R)$ is smooth and $X'\subset S(a_1,a_2,a_3)\subset\pp^{s+2}$ is hyperbolic as the union of hyperbolic varieties. The divisor class of $X'$ on $\pp(\cE)$ is $D=2H+rE$. Since $H$ is very ample and $E$ is base-point free, we have that $D$ is also very ample. Thus by Bertini's theorem, a general member of the linear system $|D|$ is smooth and irreducible \cite[Thm.~6.10]{bertini}. In particular, a sufficiently small perturbation of $X'$ in $|D|$ gives a surface with the desired properties.
\end{proof}

\begin{rem}
 The above construction does not work for $s=2$ and $r\geq2$ as in this case we must have $a_1=0$ and $a_2=a_3=1$. The image of $\pp(\cE)$ in $\pp^4$ under the map associated to the linear system $|H|$ is then $S(0,1,1)$, i.e. the cone over a quadratic hypersurface in $\pp^3$. Letting $F\subset\pp(\cE)$ be the preimage of the vertex of $S(0,1,1)$, we find that $X'.F=(2H+rE).F>1$ which implies that the image of $X'$ in $\pp^4$ cannot be smooth.
 We do not know the possible values for $r$ when $s=2$.
\end{rem}

\bigskip

\noindent \textbf{Acknowledgements.}
We would like to thank Kristin Shaw for supporting and encouraging this project and for interesting mathematical discussions. Also, we thank John Christian Ottem for helpful insights into very ampleness criteria.

 \bibliographystyle{alpha}
 \bibliography{biblio}
 \end{document}